\documentclass[final,onefignum,onetabnum]{siamart220329}

\usepackage{amsmath,amsfonts,amsopn,amssymb}
\usepackage{graphicx}
\graphicspath{ {./Figures/} }
\usepackage[caption=false]{subfig}
\usepackage{multirow,relsize,makecell}
\usepackage{url}
\ifpdf
  \DeclareGraphicsExtensions{.eps,.pdf,.png,.jpg}
\else
  \DeclareGraphicsExtensions{.eps}
\fi
\usepackage{algorithm, algorithmic}

\numberwithin{theorem}{section}
\newsiamremark{remark}{Remark}
\crefname{remark}{Remark}{Remarks}
\usepackage[normalem]{ulem}

\headers{Neuron-wise subspace correction method}{Jongho Park, Jinchao Xu, and Xiaofeng Xu}

\title{A neuron-wise subspace correction method for \\the finite neuron method\thanks{
Submitted to arXiv.
\funding{This work was supported in part by the NRF grant funded by MSIT~(No.~2021R1C1C2095193), and in part by the KAUST Baseline Research Fund.}
}}

\author{Jongho Park\thanks{Applied Mathematics and Computational Sciences Program, CEMSE, King Abdullah University of Science and Technology~(KAUST), Thuwal 23955, Saudi Arabia
 (\email{jongho.park@kaust.edu.sa}, \email{xu@multigrid.org}, \email{xiaofeng.xu@kaust.edu.sa}).}
\and
Jinchao Xu\footnotemark[2] \thanks{Department of Mathematics, Pennsylvania State University, University Park, PA 16802, USA}.
 \and
 Xiaofeng Xu\footnotemark[2]
 }

\ifpdf
\hypersetup{
  pdftitle={A neuron-wise subspace correction method for the finite neuron method},
  pdfauthor={Jongho Park, Jinchao Xu, and Xiaofeng Xu}
}
\fi

\begin{document}

\maketitle

\begin{abstract}
In this paper, we propose a novel algorithm called Neuron-wise Parallel Subspace Correction Method~(NPSC) for the finite neuron method that approximates numerical solutions of partial differential equations~(PDEs) using neural network functions.
Despite extremely extensive research activities in applying neural networks for numerical PDEs, there is still a serious lack of effective training algorithms that can achieve adequate accuracy, even for one-dimensional problems.
Based on recent results on the spectral properties of linear layers and landscape analysis for single neuron problems, we develop a special type of subspace correction method that optimizes the linear layer and each neuron in the nonlinear layer separately.
An optimal preconditioner that resolves the ill-conditioning of the linear layer is presented for one-dimensional problems, so that the linear layer is trained in a uniform number of iterations with respect to the number of neurons.
In each single neuron problem, a good local minimum that avoids flat energy regions is found by a superlinearly convergent algorithm.
Numerical experiments on function approximation problems and PDEs demonstrate better performance of the proposed method than other gradient-based methods.
\end{abstract}



\section{Introduction}
\label{Sec:Introduction}
Neural networks, thanks to the universal approximation property~\cite{Cybenko:1989,Pinkus:1999}, are promising tools for numerical solutions of partial differential equations~(PDEs). 
Moreover, it was shown in~\cite{SX:2022b} that the approximation properties of neural networks have higher asymptotic approximation rates than that of traditional numerical methods such as the finite element method.
Such powerful approximation properties, however, can hardly be observed in numerical experiments in simple tasks of function approximation despite extensive research on numerical solutions of PDEs in recent years, e.g., physics-informed neural networks~\cite{RPK:2019}, the deep Ritz method~\cite{EY:2018}, and the finite neuron method~\cite{Xu:2020}.
Even in one dimension, using gradient-based methods to training a shallow neural network does not produce accurate solutions with a substantial number of iterations. This poor convergence behavior was analyzed rigorously in~\cite{HSTX:2022} and is due to the ill-conditioning of problem.
Therefore, applying neural networks to solutions of PDEs must require novel training algorithms, different from the conventional ones for regression, image classification, and pattern recognition tasks.
In this viewpoint, there are many works on designing and analyzing training algorithms for neural networks to try to speed up the convergence or narrow the gap between the theoretical optimum and training results.
For instance, a hybrid least squares/gradient descent method was proposed in~\cite{CGPPT:2020} from an adaptive basis perspective. 
The active neuron least squares method in~\cite{AS:2022} was designed to avoid plateau phenomena that slow down the gradient dynamics of training ReLU shallow neural networks~\cite{AS:2021,PAF:2000}.
In addition, as a completely different approach, the orthogonal greedy algorithm was shown to achieve an optimal convergence rate~\cite{SX:2022a,SX:2022c}.

The aim of this paper is to develop a novel training algorithm using several recent theoretical results on training of neural networks.
A surprising recent result on training of neural networks shows that optimizing the linear layer parameters in a neural network is one bottleneck that leads to a large number of iterations of a gradient-based method.
More precisely, it was proven in~\cite{HSTX:2022} that optimizing the linear layer parameters in a ReLU shallow neural network requires solving a very ill-conditioned linear problem in general.
This work motivates us to separately design efficient solvers for the outer linear layer and the inner nonlinear layer, respectively, and train them alternately. 

Meanwhile, some recent works suggest that learning a single neuron may be a more hopeful task than learning a nonlinear layer with multiple neurons.
In~\cite{Soltanolkotabi:2017,Tian:2017,YS:2020}, convergence analyses of gradient methods for the single neuron problem with ReLU activation were presented under various assumptions on input distributions. In~\cite{VYS:2021}, the case of a single ReLU neuron with bias was analyzed. All of these results show that the global convergence of gradient methods for the single ReLU neuron problem can be attained under certain conditions.

Inspired by the above results, we consider training of the finite neuron method in one dimension as an example, and use the well-known framework of subspace correction~\cite{Xu:1992} to combine insights from the spectral properties of linear layers~\cite{HSTX:2022} and landscape analysis for single neuron problems~\cite{VYS:2021}.
Subspace correction methods provide a unified framework to design and analyze many modern iterative numerical methods such as block coordinate descent, multigrid, and domain decomposition methods.
Mathematical theory of subspace correction methods for convex optimization problems is established in~\cite{Park:2020,TX:2002}, and successful applications of it to various nonlinear optimization problems in engineering fields can be found in, e.g.,~\cite{BK:2012,LP:2019}. 
In particular, there have been successful applications of block coordinate descent methods~\cite{Wright:2015} for training of neural networks~\cite{ZLLY:2019,ZB:2017}.
Therefore, we expect that the idea of the subspace correction method is also suitable for training of neural networks for the finite neuron method.

We propose a new training algorithm called Neuron-wise Parallel Subspace Correction Method~(NPSC), which is a special type of subspace correction method for the finite neuron method~\cite{Xu:2020}.
The proposed method utilizes a space decomposition for the linear layer and each individual neuron.
In the first step of each epoch of the NPSC, the linear layer is fully trained by solving a linear system using an iterative solver.
We prove, both theoretically and numerically, that we can design an optimal preconditioner for the linear layer for one-dimensional problems, based on the relation between ReLU neural networks and linear finite elements investigated in~\cite{HLXZ:2020,HSTX:2022}.
In the second step, we train each single neuron in parallel, taking advantages of better convergence properties of learning a single neuron and a superlinearly convergent algorithm~\cite{Marquardt:1963}.
Finally, an update for the parameters in the nonlinear layer is computed by assembling the corrections obtained in the local problems for each neuron.
Due to the intrinsic parallel structure, NPSC is suitable for parallel computation on distributed memory computers. 
We present applications of NPSC to various function approximation problems and PDEs, and numerically verify that it outperforms conventional training algorithms.

The rest of this paper is organized as follows.
In \cref{Sec:FNM}, we summarize key features of the finite neuron method with ReLU shallow neural networks, and state a model problem.
An optimal preconditioner for the linear layer in one dimension is presented in \cref{Sec:Preconditioner}.
NPSC, our proposed algorithm, is presented in \cref{Sec:NPSC}.
Applications of NPSC to various function approximation problems and PDEs are presented in \cref{Sec:Numerical} to demonstrate the effectiveness of NPSC.
We conclude the paper with remarks in \cref{Sec:Conclusion}.

\section{Finite neuron method}
\label{Sec:FNM}
In this section, we introduce the finite neuron method~\cite{Xu:2020} with ReLU shallow neural networks to approximate solutions of PDEs.
We also discuss the ill-conditioning of ReLU shallow neural network~\cite{HSTX:2022}, which reveals the difficulty of training ReLU shallow neural network using gradient-based methods.

\subsection{Model problem} 
We consider the following model problem:
 \begin{equation}
 \label{model_cont}
 \min_{u \in V} \left\{ \frac{1}{2} a(u, u) - \int_{\Omega} f u \,dx \right\},
 \end{equation}
where $\Omega \subset \mathbb{R}^d$ is a bounded domain, $f \in L^2 (\Omega)$, $a(\cdot, \cdot)$ is a continuous, coercive, and symmetric bilinear form defined on a Hilbert space $V \subset L^2 (\Omega)$.
Various elliptic boundary value problems can be formulated as optimization problems of the form~\eqref{model_cont}~\cite{Xu:2020}.

A ReLU shallow neural network with $n$ neurons is given by
\begin{equation}
\label{NN}
u(x; \theta) = \sum_{i=1}^n a_i \sigma (\omega_i \cdot x + b_i), \quad
\theta = \left\{ a, \omega, b \right\} = \left\{ ( a_i )_{i=1}^n, ( \omega_i )_{i=1}^n, ( b_i )_{i=1}^n \right\},
\end{equation}
where $x \in \mathbb{R}^d$ is an input, $\theta$ is the collection of parameters consisting of $a_i \in \mathbb{R}$, $\omega_i \in \mathbb{R}^d$, and $b_i \in \mathbb{R}$ for $1 \leq i \leq n$, and $\sigma \colon \mathbb{R} \rightarrow \mathbb{R}$ is the ReLU activation function defined by $\sigma (x) = \max \{ 0,x \}$.
The neural network~\eqref{NN} possesses a total of $(d+2)n$ parameters.

The collection of all neural network functions of the form~\eqref{NN} is denoted by $\Sigma_n$, i.e.,
\begin{equation*}
\Sigma_n = \left\{ v(x) = \sum_{i=1}^n a_i \sigma (\omega_i \cdot x + b_i) : a_i \in \mathbb{R}, \hspace{0.1cm} \omega_i \in \mathbb{R}^d, \hspace{0.1cm} b_i \in \mathbb{R}\right\}.
\end{equation*}
The space $\cup_{n \in \mathbb{N}}\Sigma_n$ enjoys the universal approximation property~\cite{Cybenko:1989,Pinkus:1999}, namely, any function with sufficient regularity can be uniformly approximated by functions in $\cup_{n \in \mathbb{N}}\Sigma_n$.
Recent results on such approximation properties can be found in~\cite{SX:2022b,SX:2022c}.

In the finite neuron method, we consider the Galerkin approximation of~\eqref{model_cont} on the space $\Sigma_n$:
\begin{equation*}
    \min_{u \in \Sigma_n} \left\{ \frac{1}{2} a(u,u) - \int_{\Omega} fu \,dx \right\},
\end{equation*}
which is equivalent to
\begin{equation}
 \label{model}
 \min_{\theta \in \Theta} \left\{ E(\theta) := \frac{1}{2} a \left( u(x; \theta), u(x; \theta) \right) - \int_{\Omega} f(x) u(x; \theta) \,dx \right\},
\end{equation}
where $\Theta = \mathbb{R}^{(d+2)n}$ is the parameter space.
That is, in~\eqref{model}, we find a good approximation of the solution of the continuous problem~\eqref{model_cont} in terms of the neural network~\eqref{NN}.

\subsection{Ill-conditioning of the linear layer}
From the structure of the neural network~\eqref{NN}, we say that the parameter $a$ belongs to the linear layer and $\omega$ and $b$ belong to the nonlinear layer.
In this subsection, we explain that even optimizing parameters in the linear layer becomes a bottleneck if we use simple gradient-based methods. A large number of iterations is needed to achieve satisfatory accuracy due to the ill-conditioning of the linear layer. 

\begin{figure}
  \centering
  \includegraphics[width=0.69\hsize]{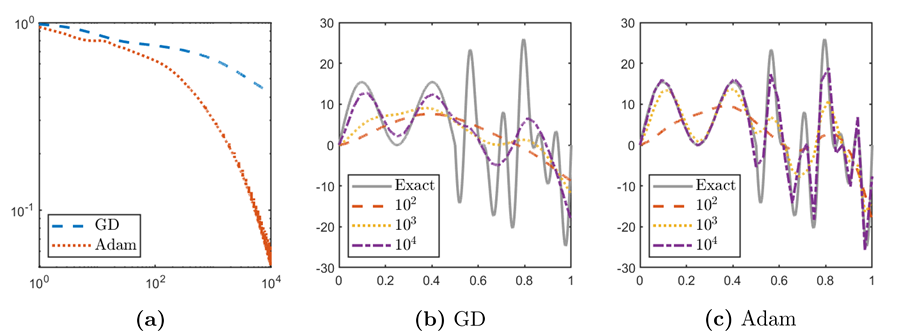}
  \caption{Numerical results for the function approximation problem~\eqref{M}. 
  {\rm\textbf{(a)}} Decay of the relative energy error $\frac{E_M(a^{(k)}) - E_M ( M^{-1}\beta)}{|E_M (M^{-1} \beta)|}$ in the gradient descent method~(GD) and Adam, where $k$ denotes the number of iterations.
  {\rm\textbf{(b, c)}} Exact solution and its approximations generated by various numbers of iterations of GD and Adam~($n = 2^5$).}
  \label{Fig:M}
\end{figure}

To illustrate the ill-conditioning, we consider  $\Omega = (0,1) \subset \mathbb{R}$ and the following bilinear form
\begin{equation}
    \label{L2}
    a(u,v) = \int_{\Omega} uv \,dx, \quad u,v \in V, \quad V= L^2 (\Omega)
\end{equation}
in~\eqref{model}.
We obtain the $L^2$-function approximation problem, which is the most elementary instance of~\eqref{model}.
We fix the parameters $\omega$ and $b$ in~\eqref{model} as follows:
\begin{equation*}
    \omega_i = 1, \quad
    b_i = - \frac{i-1}{n}, \quad
    1 \leq i \leq n.
\end{equation*}
Then~\eqref{model} is written as
\begin{equation}
    \label{M}
    \min_{a \in \mathbb{R}^n} \left\{ E_M (a) := \frac{1}{2} a^{\mathrm{T}} M a - \beta^{\mathrm{T}} a \right\},
\end{equation}
where $M \in \mathbb{R}^{n \times n}$ and $\beta \in \mathbb{R}^n$ are given by
\begin{equation*}
    \begin{aligned}
    &M_{ij} = \int_{\Omega} \sigma \left( x - \frac{j-1}{n} \right) \sigma \left( x - \frac{i-1}{n} \right) \,dx, \\
    &\beta_i = \int_{\Omega} f(x) \sigma \left( x - \frac{i-1}{n} \right) \,dx,
    \end{aligned}
    \quad 1 \leq i, j \leq n.
\end{equation*}

When we solve the minimization problem~\eqref{M} by a gradient-based method such as the gradient descent method or Adam~\cite{KB:2015}, the convergence rate depends on the condition number $\kappa (M)$ of $M$. It was recently proved in~\cite[Theorem 1]{HSTX:2022} that $M$ is very ill-conditioned. This result is stated in the following proposition. 

\begin{proposition}
\label{Prop:M}
In~\eqref{M}, the condition number $\kappa (M)$ of the matrix $M$ satisfies
\begin{equation*}
    \kappa (M) = O(n^4).
\end{equation*}
\end{proposition}

The condition number $\kappa(M)$ becomes exceedingly large when $n$ increases. The convergence of simple gradient-based methods thus becomes extremely slow. We now demonstrate the poor convergence of the gradient descent method and Adam for solving~\eqref{M} with $f$ given by
\begin{equation}
\label{Ex2_pre}
f(x) = \begin{cases}
10 \left( \sin 2\pi x + \sin 6 \pi x \right), & \quad \textrm{ if } 0 < x < \frac{1}{2}, \\
10 \left( \sin 8\pi x + \sin 18 \pi x + \sin 26 \pi x \right), & \quad \textrm{ if } \frac{1}{2} \leq x < 1.
\end{cases}
\end{equation}
In both methods, the learning rate $\tau$ is chosen as $\tau = 2/(\lambda_{\min} (M) + \lambda_{\max} (M))$, which is optimal in the sense of~\cite[Lemma~C.5]{TW:2005}; $\lambda_{\min}$ and $\lambda_{\max}$ stand for the minimum and maximum eigenvalues, respectively.
We use the zero initial guess.
One can observe in \cref{Fig:M}(a) that the convergence rates of both methods are fairly slow even when $n$ is not large; the relative energy error $\frac{E_M(a^{(k)}) - E_M ( M^{-1}\beta)}{|E_M (M^{-1} \beta)|}$ does not reach $10^{-2}$ with $10^4$ iterations.
Moreover, as shown in \cref{Fig:M}(b, c), both methods give poor numerical approximations to $f$ despite large numbers of iterations.

\section{Optimal preconditioner for the linear layer in one dimension}
\label{Sec:Preconditioner}
In this section, we propose an optimal preconditioner for the linear layer in one dimension.
Using the proposed preconditioner, the linear layer can be fully trained within a uniform number of iterations with respect to the number of neurons $n$. 

For simplicity, we set $\Omega = (0, 1) \subset \mathbb{R}$ in~\eqref{model}.
Fixing the parameters $\omega$ and $b$ in the nonlinear layer, \eqref{model} reduces to the following minimization problem with respect to $a$:
\begin{equation}
\label{a_min}
    \min_{a \in \mathbb{R}^n} \left\{ \frac{1}{2} a^{\mathrm{T}} K a - \beta^{\mathrm{T}} a \right\},
\end{equation}
where $K \in \mathbb{R}^{n \times n}$ and $\beta \in \mathbb{R}^n$ are given by
\begin{equation}\label{stiffness_relu}
\begin{aligned}
&K_{ij} = a \left( \sigma (\omega_j x + b_j), \sigma (\omega_i x + b_i) \right) , \\
&\beta_i = \int_{\Omega} f(x) \sigma ( \omega_i x + b_i ) \,dx,
\end{aligned}
\quad 1 \leq i, j \leq n.
\end{equation}
Let
\begin{equation*}
x_i = -\frac{b_i}{\omega_i}, \quad 1 \leq i \leq n,
\end{equation*}
be the nodal point determined by the ReLU function $\psi_i (x) = \sigma (\omega_i x + b_i )$.
We denote the number of the nodal points inside $\Omega$ by $n_{\Omega}$.
Without loss of generality, we assume the following:
\begin{subequations}
\label{nodes}
\begin{align}
\label{nodes1}
& \psi_i \not\equiv 0 \text{ on }\Omega, \quad 1 \leq i \leq n, \\
\label{nodes2}
&n_{\Omega} \geq n - 2, \\
\label{nodes3}
&x_1, \dots, x_{n_{\Omega}} \in \Omega \hspace{0.1cm}\text{ and }\hspace{0.1cm} x_1 < \dots < x_{n_{\Omega}}, \\
\label{nodes4}
&x_{n_{\Omega + 1}}, \dots, x_n \in \mathbb{R} \setminus \Omega \hspace{0.1cm}\text{ and }\hspace{0.1cm} x_{n_{\Omega + 1}} < \dots < x_{n}.
\end{align}
\end{subequations}
The assumption~\eqref{nodes1} requires that no $\psi_i$ vanishes on $\Omega$ because otherwise they do not contribute to the minimization problem~\eqref{a_min}. The assumption~\eqref{nodes2} requires at most two neurons with nodal points outside $\Omega$. This is because the corresponding ReLU functions become linearly dependent on $\Omega$ if there are more than two neurons with nodal points outside $\Omega$. Finally, the assumptions~\eqref{nodes3} and~\eqref{nodes4} can be satisfied under an appropriate reordering.

Under~\eqref{nodes}, it is proved in \cite[Theorem~2.1]{HLXZ:2020} that $\{ \psi_i \}_{i=1}^n$ are linearly independent on $\Omega$. It is easy to see $K$ is symmetric and positive definite~(SPD).
Let $\mathcal{A} \colon V \rightarrow V$ be a bijective linear operator such that
\begin{equation}
    \label{cA}
    a(u,v) = \int_{\Omega} (\mathcal{A} u) v \,dx, \quad u,v \in V.
\end{equation}
Writing $\Psi = [\psi_1, \dots, \psi_n ]^{\mathrm{T}}$, we have
\begin{equation*}
    K = \int_{\Omega} (\mathcal{A} \Psi (x)) \Psi (x)^{\mathrm{T}} \,dx,
\end{equation*}
where $\mathcal{A}$ is applied entrywise.

Before we present the optimal preconditioner, we state the following lemma that is used in the construction of the proposed preconditioner.

\begin{lemma}
\label{Lem:preconditioner}
For two positive integers $m \geq n$, let $A$, $B \in \mathbb{R}^{m \times m}$ be two SPD matrices and let $R \in \mathbb{R}^{n \times m}$ be a surjective matrix.
Then $R B^{-1} R^{\mathrm{T}}$ is SPD and
\begin{equation*}
    \kappa \left( \left(R B^{-1} R^{\mathrm{T}} \right)^{-1} R A R^{\mathrm{T}} \right) \leq \kappa (B A).
\end{equation*}
\end{lemma}
\begin{proof}
    It is easy to see that $B^{-1}$ is SPD and $R^{\mathrm{T}}$ is injective.
    Then we have
\begin{equation*}
    \langle R B^{-1} R^{\mathrm{T}} \alpha, \alpha \rangle = \langle B^{-1} R^{\mathrm{T}} \alpha, R^\mathrm{T} \alpha \rangle > 0,\quad \forall \alpha \neq \mathbf{0}. 
\end{equation*}
Hence, $R B^{-1} R^T$ is also SPD.

One can establish a lower bound for the minimum eigenvalue of the matrix $(RB^{-1}R^{\mathrm{T}})^{-1} RAR^{\mathrm{T}}$ as follows:
\begin{multline}
\label{lambda_min}
    \lambda_{\min} \left( \left(R B^{-1} R^{\mathrm{T}} \right)^{-1} R A R^{\mathrm{T}} \right)
    = \min_{\alpha \neq \mathbf{0}} \frac{\alpha^{\mathrm{T}} R A R^{\mathrm{T}} \alpha}{\alpha^{\mathrm{T}} R B^{-1} R^{\mathrm{T}} \alpha}  \\
    = \min_{\substack{\bar{\alpha} \in \operatorname{ran} R^{\mathrm{T}}, \\ \bar{\alpha} \neq \mathbf{0}} } \frac{\bar{\alpha}^{\mathrm{T}} A \bar{\alpha}}{\bar{\alpha}^{\mathrm{T}} B^{-1} \bar{\alpha}} 
    \geq \min_{\bar{\alpha} \neq \mathbf{0} } \frac{\bar{\alpha}^{\mathrm{T}} A \bar{\alpha}}{\bar{\alpha}^{\mathrm{T}} B^{-1} \bar{\alpha}} = \lambda_{\min} (BA).
\end{multline}
In the same manner, we have
\begin{equation}
\label{lambda_max}
\lambda_{\max} \left( \left(R B^{-1} R^{\mathrm{T}} \right)^{-1} R A R^{\mathrm{T}} \right) \leq \lambda_{\max} (B A).
\end{equation}
Combining~\eqref{lambda_min} and~\eqref{lambda_max} yields the desired result.
\end{proof}

Now we are ready to construct our proposed preconditioner, which resolves the issue of ill-conditioning explained in \cref{Sec:FNM}.
We first outline the main idea for the construction of the preconditioner.
By suitably enriching the space spanned by the ReLU functions $ \{ \psi_i \}_{i=1}^n$, an alternative set of basis consisting of hat functions is available.
Applying a change of variables to these hat basis functions transforms the system matrix into a form amenable to direct solution with $\mathcal{O}(n)$ arithmetic operations.
Consequently, by combining this change of variables with a direct solver for the transformed matrix, we obtain the desired preconditioner.

We define
\begin{equation*}
\psi_{\mathrm{L}} (x) = x, \quad \psi_{\mathrm{R}} (x) = 1-x,
\end{equation*}
and write $\overline{\Psi} = [\psi_1 , \dots, \psi_{n_{\Omega}} , \psi_{\mathrm{L}}, \psi_{\mathrm{R}}]^{\mathrm{T}}$.
That is, $\overline{\Psi}$ is formed by augmenting the ReLU functions $\{ \psi_i \}_{i=1}^{n_{\Omega}}$ with additional ReLU functions whose nodes are the located at the endpoints of the domain $\Omega$.
Since each $\psi_i$, $n_{\Omega} + 1 \leq i \leq n$, is linear on $\Omega$, we have
\begin{equation*}
    \psi_i = \psi_i (1) \psi_{\mathrm{L}} + \psi_i (0) \psi_{\mathrm{R}} \hspace{0.1cm}\text{ on }\hspace{0.1cm} \Omega.
\end{equation*}
Hence, we readily get $\Psi = R \overline{\Psi}$ on $\Omega$, where
\begin{equation}
\label{R}
R = \begin{bmatrix}
I_{n_{\Omega}} & \mathbf{0} \\
\mathbf{0} & \widetilde{R} \end{bmatrix}
\in \mathbb{R}^{n \times (n_{\Omega} + 2)}, \quad
\widetilde{R} = \begin{bmatrix}
\psi_{n_{\Omega} + 1} (1) & \psi_{n_{\Omega} + 1} (0) \\
\vdots & \vdots \\
\psi_{n} (1) & \psi_{n} (0) \end{bmatrix}
\in \mathbb{R}^{(n - n_{\Omega}) \times 2}.
\end{equation}
In~\eqref{R}, $I_{n_{\Omega}}$ denotes the $\mathbb{R}^{n_{\Omega} \times n_{\Omega}}$ identity matrix.

We define
\begin{equation}
    \label{bM}
    \overline{K} = \int_{\Omega} (\mathcal{A} \overline{\Psi} (x)) \overline{\Psi} (x)^{\mathrm{T}} \,dx.
\end{equation}
It is easy to see $K = R \overline{K} R^{\mathrm{T}}$.
Invoking~\cite[Theorem~2.1]{HLXZ:2020} implies that the entries of $\overline{\Psi}$ are linearly independent, so that they form a basis for the space $V_{n_{\Omega} + 2}$ of continuous and piecewise linear functions on the grid $0 < x_1 < \dots < x_{n_{\Omega}} < 1$.

If we set $\overline{\Psi}^+ = [\psi_1^+, \dots, \psi_{n_{\Omega}}^+, \psi_{\mathrm{L}}, \psi_{\mathrm{R}}]^{\mathrm{T}}$, where
\begin{equation*}
    \psi_i^+ (x) = \sigma (x - x_i), \quad 1 \leq i \leq n_{\Omega},
\end{equation*}
then by direct calculation we get
\begin{equation}
\label{B1}
    \psi_i = \begin{cases}
     \omega_i  \psi_i^+, & \quad \text{ if } \omega_i \geq 0, \\
    \omega_i ( 1 - x_i) \psi_{\mathrm{L}} - \omega_i \psi_i^+ - \omega_i x_i \psi_{\mathrm{R}}, & \quad \text{ otherwise.}
    \end{cases}
\end{equation}
Note that $\overline{\Psi}^+$ is formed by replacing the slopes of the ReLU functions $\{ \psi_i \}_{i=1}^{n_{\Omega}}$ with $1$.
Using~\eqref{B1}, it is straightforward to construct a matrix $B_1 \in \mathbb{R}^{(n_{\Omega}+2) \times (n_{\Omega}+2)}$ satisfying
\begin{equation}
\label{B1_relation}
    \overline{\Psi} = B_1 \overline{\Psi}^+.
\end{equation}
The matrix $B_1$ is nonsingular since both the entries of $\overline{\Psi}$ and those of $\overline{\Psi}^+$ form bases for the space $V_{n_{\Omega + 2}}$~\cite{HLXZ:2020}.

Meanwhile, $V_{n_{\Omega} + 2}$ admits the standard hat basis $\{ \phi_1, \dots, \phi_{n_{\Omega}}, \phi_{\mathrm{L}}, \phi_{\mathrm{R}}  \}$, which is given by
\begin{align*}
&\phi_i (x) = \begin{cases}
\dfrac{x - x_{i-1}}{x_i - x_{i-1}}, &
\quad \text{ if } x \in (x_{i-1}, x_i], \\
\dfrac{x - x_{i+1}}{x_i - x_{i+1}}, &
\quad \text{ if } x \in [x_i, x_{i+1}), \\
0, & \quad \text{ otherwise,}
\end{cases}
 \quad 1 \leq i \leq n_{\Omega}, \\
 &\phi_{\mathrm{L}} (x) = \begin{cases}
 -\dfrac{x - x_1}{x_1}, &
 \quad \text{ if } x \in (0, x_1], \\
 0, & \quad \text{ otherwise,}
 \end{cases}
 \quad
 \phi_{\mathrm{R}} (x) = \begin{cases}
 \dfrac{x - x_{n_{\Omega}}}{1 - x_{n_{\Omega}}}, &
 \quad \text{ if } x \in [x_{n_{\Omega}}, 1), \\
 0, & \quad \text{ otherwise,}
 \end{cases}
\end{align*}
with abuse of notation $x_{-1} = 0$ and $x_{n_{\Omega} + 1} = 1$.
Similar to~\eqref{bM}, we write $\overline{\Phi} = [\phi_1, \dots, \phi_{n_{\Omega}}, \phi_{\mathrm{L}}, \phi_{\mathrm{R}}]^{\mathrm{T}}$ and set
\begin{equation*}
    \overline{K}_{\phi} = \int_{\Omega} (\mathcal{A} \overline{\Phi} (x)) \overline{\Phi} (x)^{\mathrm{T}} \,dx.
\end{equation*}
One can verify the following relation between the entries of $\overline{\Phi}$ and those of $\overline{\Psi}^+$ by direct calculation:
\begin{equation}
\label{B2}
\begin{split}
    \phi_1 &= \frac{1}{x_1} \psi_{\mathrm{L}} - \frac{x_2}{(x_2 - x_1) x_1} \psi_1^+ + \frac{1}{x_2 - x_1} \psi_2^+,\\
    \phi_i &= \frac{1}{x_i - x_{i-1}} \psi_{i-1}^+ - \frac{x_{i+1} - x_{i-1}}{(x_{i+1} - x_i)(x_i - x_{i-1})} \psi_i^+ 
    + \frac{1}{x_{i+1}-x_i} \psi_{i+1}^+ , \hspace{0.1cm} 2 \leq i \leq n_{\Omega}-1,\\
    \phi_{n_{\Omega}} &= \frac{1}{x_{n_{\Omega}} - x_{n_{\Omega}-1}} \psi_{n_{\Omega}-1}^+ - \frac{1 - x_{n_{\Omega}-1}}{(1-x_{n_{\Omega}})(x_{n_{\Omega}}-x_{n_{\Omega}-1})} \psi_{n_{\Omega}}^+,\\
    \phi_{\mathrm{L}} &= - \frac{1-x_1}{x_1} \psi_{\mathrm{L}} + \frac{1}{x_1} \psi_1^+ + \psi_{\mathrm{R}}, \\
    \phi_{\mathrm{R}} &= \frac{1}{1 - x_n} \psi_{n_{\Omega}}^+.
\end{split}
\end{equation}
Hence, we can construct a matrix $B_2 \in \mathbb{R}^{(n_{\Omega}+2) \times (n_{\Omega}+2)}$ such that
\begin{equation}
    \label{B2_relation}
    \overline{\Phi} = B_2 \overline{\Psi}^+
\end{equation}
explicitly using~\eqref{B2}.
Combining~\eqref{B1_relation} and~\eqref{B2_relation} yields
\begin{equation}
\label{C_relation}
    \overline{\Phi} = B \overline{\Psi},
\end{equation}
where $B = B_2 B_1^{-1}$.
\Cref{C_relation} implies that two matrices $\overline{K}$ and $\overline{K}_{\phi}$ are related as follows:
\begin{equation}
\label{K_relation}
    \overline{K} = \int_{\Omega} (\mathcal{A} \overline{\Psi} (x)) \overline{\Psi} (x)^{\mathrm{T}} \,dx
    = \int_{\Omega} B^{-1}( \mathcal{A} \overline{\Phi} (x)) \overline{\Phi} (x)^{\mathrm{T}} B^{-\mathrm{T}} \,dx
    = B^{-1} \overline{K}_{\phi} B^{-\mathrm{T}}.
\end{equation}
Since $K = R \overline{K} R^{\mathrm{T}}$, invoking \cref{Lem:preconditioner}, setting $\overline{P} = B^{\mathrm{T}} \overline{K}_{\phi}^{-1} B$ and
\begin{equation}
\label{P}
    P = \left( R \overline{P}^{-1} R^{\mathrm{T}} \right)^{-1}
\end{equation}
completes the construction of the proposed preconditioner, where $R$ was defined in~\eqref{R}.
We summarize our main result described above in the following theorem.

\begin{theorem}
\label{Thm:preconditioner}
Let $\Omega = (0,1) \subset \mathbb{R}$ and let $K \in \mathbb{R}^{n \times n}$ be the matrix defined in~\eqref{stiffness_relu}.
Assume that~\eqref{nodes} holds.
Then the preconditioner $P \in \mathbb{R}^{n \times n}$ given in~\eqref{P} satisfies
\begin{equation*}
    \kappa (P K) = O(1),
\end{equation*}
i.e., $\kappa (P K)$ has an upper bound independent of $n$, $\omega$, and $b$.
\end{theorem}

The proposed preconditioner $P$ can be implemented very efficiently, requiring only $\mathcal{O} (n)$ elementary arithmetic operations.
In particular, solving a linear system represented by the matrix $\overline{K}_{\phi}$ can be accomplished with $\mathcal{O} (n)$ elementary arithmetic operations.
Details of the implementation of the proposed preconditioner are discussed in \cref{App:Implementation}.    

\begin{remark}
\label{Rem:diagonal}
In the case of the $L^2$-function approximation problem~\eqref{L2}, we can construct an alternative preconditioner $P_{\operatorname{diag}}$ whose computational cost is a bit cheaper than $P$.
Let $\overline{P}_{\operatorname{diag}} = B^{\mathrm{T}} \operatorname{diag} (\overline{K}_{\phi})^{-1} B$.
It follows from~\eqref{K_relation} that
\begin{equation*}
    \overline{P}_{\operatorname{diag}} \overline{K} = B^{\mathrm{T}} \operatorname{diag} (\overline{K}_{\phi})^{-1} B \overline{K} = B^{\mathrm{T}} \operatorname{diag} (\overline{K}_{\phi})^{-1} \overline{K}_{\phi} B^{-\mathrm{T}}.
\end{equation*}
In~\eqref{L2}, $\overline{K}_{\phi}$ is a mass matrix for the linear finite element method defined on the grid $0 < x_1 < \dots < x_{n_{\Omega}} < 1$, so that it satisfies $\kappa (\operatorname{diag} (\overline{K}_{\phi})^{-1} \overline{K}_{\phi}) = O(1)$.
Then we have
\begin{equation*}
    \kappa(\overline{P}_{\operatorname{diag}} \, \overline{K}) = \kappa(B^{\mathrm{T}} \operatorname{diag} (\overline{K}_{\phi})^{-1} \overline{K}_{\phi} B^{-\mathrm{T}}) = \kappa(\operatorname{diag} (\overline{K}_{\phi})^{-1} \overline{K}_{\phi}) = O(1).
\end{equation*}
Therefore, $P_{\operatorname{diag}} = ( R \overline{P}_{\operatorname{diag}}^{-1} R^{\mathrm{T}})^{-1}$ satisfies $\kappa (P_{\operatorname{diag}} K) = O(1)$ by \cref{Lem:preconditioner}.
It is evident that the computational cost of $\operatorname{diag} (\overline{K}_{\phi})^{-1}$ is cheaper than that of $\overline{K}_{\phi}^{-1}$.
\end{remark}

\section{Neuron-wise Parallel Subspace Correction Method~(NPSC)}
\label{Sec:NPSC}

\begin{figure}
  \centering
  \includegraphics[width=0.90\hsize]{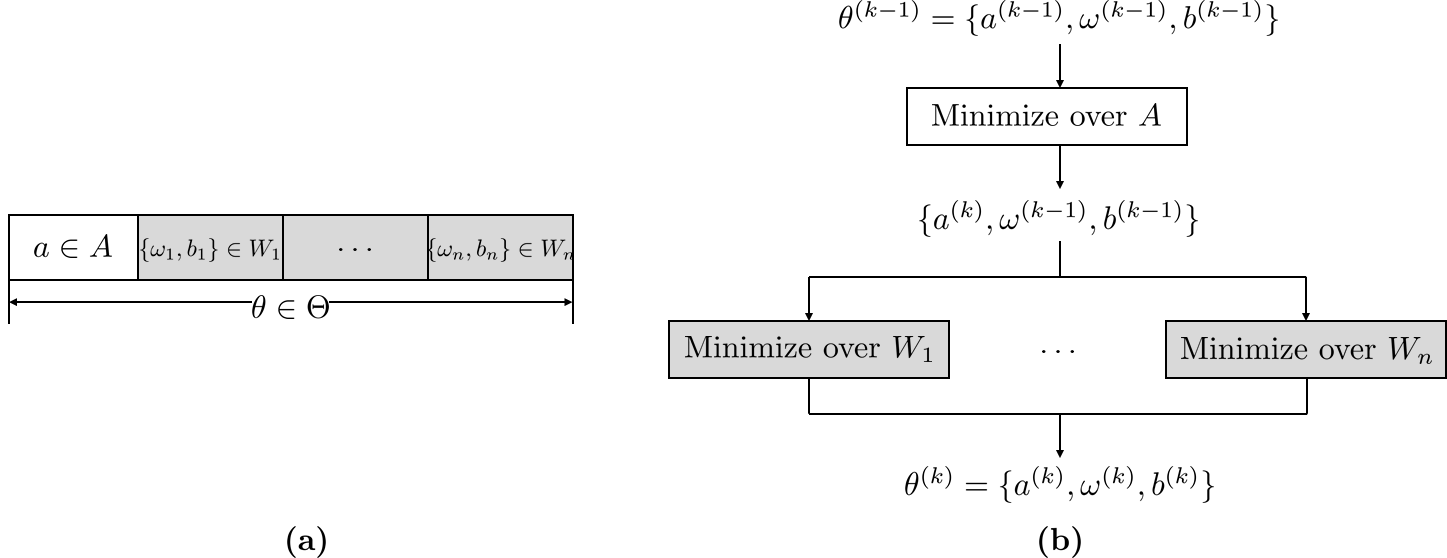}
  \caption{{\rm\textbf{(a)}} Space decomposition of the solution space $\Theta$ of~\eqref{model} into subspaces $A$ and $\{ W_i \}_{i=1}^n$.
  {\rm\textbf{(b)}} Subspace correction procedure of NPSC.}
  \label{Fig:space_decomp}
\end{figure}

In this section, we introduce NPSC, a subspace correction method~\cite{Xu:1992} that deals with the linear layer and each neuron in the nonlinear layer separately for solving~\eqref{model}.
Thanks to the preconditioner proposed in the previous section, the linear layer of the neural network~\eqref{NN} can be fully trained with a cheap computational cost. 

We first present a space decomposition for NPSC.
The parameter space $\Theta$ admits a natural decomposition $\Theta = A \oplus W$, where $A = \mathbb{R}^n$ and $W = \mathbb{R}^{(d+1)n}$ are the spaces for $a$ and $\{\omega, b\}$, respectively, and $\oplus$ denotes direct sum.
Since any $\{\omega, b\} \in W$ consists of the parameters $(\{ \omega_i, b_i \})_{i=1}^n$ from $n$ neurons, $W$ can be further decomposed as $W = \bigoplus_{i=1}^n W_i$, where $W_i = \mathbb{R}^{d+1}$ is the space for $\{\omega_i, b_i \}$.
Finally, we have the following space decomposition of $\Theta$:
\begin{equation}
\label{space_decomp}
\Theta = A \oplus \bigoplus_{i=1}^n W_i.
\end{equation}
A graphical description for the space decomposition~\eqref{space_decomp} is presented in \cref{Fig:space_decomp}(a).

\begin{algorithm}
\caption{Neuron-wise Parallel Subspace Correction Method~(NPSC) for~\eqref{model}}
\begin{algorithmic}[]
\label{Alg:NPSC}
\STATE Choose an initial guess $\theta^{(0)} = \{ a^{(0)}, \omega^{(0)}, b^{(0)} \}$ and an initial learning rate $\tau_0 = 1$.
\FOR{$k=0,\dots, T-1$}
\STATE Adjust $\{ \omega^{(k)}, b^{(k)} \}$ to avoid linear dependence of the neurons~(see \cref{Alg:adj}).
\STATE $\displaystyle
a^{(k+1)} \in \operatornamewithlimits{\arg\min}_{a \in A} E (\{ a, \omega^{(k)}, b^{(k)} \})
$~(see~\eqref{a_min_linear})
\FOR{$i = 1, \dots, n$ \textbf{in parallel}}
\STATE 
\begin{equation}
    \label{local}
\{ \omega_i^{(k+\frac{1}{2})}, b_i^{(k+\frac{1}{2})} \} \in \operatornamewithlimits{\arg\min}_{\{\omega_i, b_i \} \in W_i} E \left( \left\{ a^{(k+1)}, \omega_i \oplus \bigoplus_{j\neq i}  \omega_j^{(k)} , b_i \oplus \bigoplus_{j\neq i}  b_j^{(k)} \right\} \right)
\end{equation}
\ENDFOR
\STATE Determine the learning rate $\tau_k$ by backtracking~(see \cref{Alg:backt}).
\FOR{$i = 1, \dots, n$ \textbf{in parallel}}
\STATE $\displaystyle
\omega_i^{(k+1)} = (1- \tau_k) \omega_i^{(k)} + \tau_k \omega_i^{(k+\frac{1}{2})}
$
\STATE $\displaystyle
b_i^{(k+1)} = (1- \tau_k) b_i^{(k)} + \tau_k b_i^{(k+\frac{1}{2})}
$
\ENDFOR
\ENDFOR
\end{algorithmic}
\end{algorithm}

NPSC, our proposed method, is presented in \cref{Alg:NPSC}.
It is a subspace correction method~\cite{Xu:1992} for~\eqref{model} based on the space decomposition~\eqref{space_decomp}. At the $k$th epoch, NPSC updates the parameter $a$ first by minimizing the loss function with respect to $a$, then it updates the parameters in each neuron by minimizing $E$ with respect to $\{ \omega_i, b_i \}$ in parallel.
The update of $\{ \omega_i, b_i \}$ is relaxed by an appropriate learning rate $\tau_k > 0$ as in the existing parallel subspace correction methods for optimization problems~\cite{LP:2019,TX:2002}.
The overall structure of NPSC is depicted in \cref{Fig:space_decomp}(b).
The space decomposition allows flexible usage of solvers for each subproblems.
In the remainder of this section, we discuss different parts of \cref{Alg:NPSC} in detail.

\subsection{\texorpdfstring{$a$}{a}-minimization problem}
\label{Subsec:a}
As we discussed in \cref{Sec:Preconditioner}, the $a$-minimization problem in \cref{Alg:NPSC} can be represented as a quadratic optimization problem of the form~\eqref{a_min}.
Equivalently, it suffices to solve a system of linear equations
\begin{equation}
    \label{a_min_linear}
    K a = \beta,
\end{equation}
where $K$ and $\beta$ are defined in a similar manner as~\eqref{stiffness_relu}.
Since it is guaranteed by \cref{Alg:adj} that $K$ is SPD,~\eqref{a_min_linear} can be solved by the preconditioned conjugate gradient method~(see, e.g.,~\cite{Saad:2003}).
Equipped with the preconditioner proposed in \cref{Sec:Preconditioner}, the preconditioned conjugate gradient method always finds a solution of~\eqref{a_min_linear} up to a certain level of precision within a uniform number of iterations with respect to $n$, $\omega^{(k)}$, and $b^{(k)}$.

\subsection{\texorpdfstring{$\{ \omega_i, b_i \}$}{(omegai, bi)}-minimization problems}
\label{Subsec:Local}
Training the nonlinear layer presents challenges due to its nonconvexity and complex landscape.
In the following, we discuss the training process of the nonlinear layer in NPSC, where each neuron is trained separately.

We consider the $\{ \omega_i, b_i \}$-minimization problem~\eqref{local} in \cref{Alg:NPSC} for a fixed $i$.
We may assume that $a_i^{(k+1)} \neq 0$; otherwise, the minimization problem becomes trivial.
Under some elementary manipulations,~\eqref{local} is rewritten as
\begin{equation}
\label{local_general}
\min_{\{ \omega_i, b_i \} \in \mathbb{R}^{d+1}} \left\{ E_{i} (\omega_i, b_i) := \frac{1}{2} a \left( \psi_i, \psi_i \right) - \int_{\Omega} F \psi_i \,dx \right\},
\end{equation}
where $\psi_i (x) = \sigma (\omega_i \cdot x + b_i)$ and $F \in L^2 (\Omega)$ is a function determined in terms of $f$, $a^{(k+1)}$, $\omega_j^{(k)}$, and $b_j^{(k)}$ for $j \neq i$.
Indeed, we have
\begin{equation*} \begin{split}
E &\left( \left\{ a^{(k+1)}, \omega_i \oplus \bigoplus_{j\neq i}  \omega_j^{(k)} , b_i \oplus \bigoplus_{j\neq i}  b_j^{(k)} \right\} \right) \\
&= \frac{1}{2} a \left( a_i^{(k+1)} \psi_i + \sum_{j \neq i} a_j^{(k+1)} \psi_j^{(k)}, a_i^{(k+1)} \psi_i + \sum_{j \neq i} a_j^{(k+1)} \psi_j^{(k)} \right) \\
&\quad- \int_{\Omega} f \left( a_i^{(k+1)} \psi_i + \sum_{j \neq i} a_j^{(k+1)} \psi_j^{(k)} \right) \,dx \\
&= (a_i^{(k+1)})^2 \left( \frac{1}{2} a(\psi_i, \psi_i) + \frac{1}{a_i^{(k+1)}} a \left(  \sum_{j \neq i} a_j^{(k+1)} \psi_j^{(k)}, \psi_i \right) - \frac{1}{a_i^{(k+1)}} \int_{\Omega} f \psi_i \,dx \right) + \text{constant} \\
&\stackrel{\eqref{cA}}{=} (a_i^{(k+1)})^2 \left( \frac{1}{2} a(\psi_i, \psi_i) - \frac{1}{a_i^{(k+1)}} \int_{\Omega} \left( f - \mathcal{A} \left(  \sum_{j \neq i} a_j^{(k+1)} \psi_j^{(k)} \right) \right) \psi_i \,dx \right) + \text{constant},
\end{split} \end{equation*}
which proves~\eqref{local_general} with
\begin{equation*}
    F = \frac{1}{a_i^{(k+1)}} \left( f - \mathcal{A} \left(  \sum_{j \neq i} a_j^{(k+1)} \psi_j^{(k)} \right) \right),
\end{equation*}
 where $\psi_j^{(k)}(x) = \sigma (\omega_j^{(k)} \cdot x + b_j^{(k)})$ for $j \neq i$.

Given that single neuron training~\eqref{local_general} is relatively simple compared to training the entire nonlinear layer, there have been efforts to rigorously analyze the landscape of single neuron training~\cite{VYS:2021,YS:2020}.
We present an existing result~\cite[Theorem~4.2]{VYS:2021} on single neuron training~\eqref{local_general} under a particular assumption on $F$.

\begin{proposition}
\label{Prop:Vardi}
In the single neuron problem~\eqref{local_general}, we assume that
\begin{equation}
    \label{Vardi}
    F(x) = \sigma (\hat{\omega}_i \cdot x + \hat{b}_i )
\end{equation}
for some $\hat{\omega}_i \in \mathbb{R}^d$ and $\hat{b}_i \in \mathbb{R}$.
Then any critical point $\{ \omega_i^*, b_i^* \} \in \mathbb{R}^{d+1}$ such that $\sigma (\omega_i^* \cdot x + b_i^*) \not\equiv 0$ on $\Omega$ is a global minimum.
\end{proposition}

\Cref{Prop:Vardi} explains that the landscape of the single neuron problem~\eqref{local_general} is rather manageable: it comprises a flat region of parameters that make the neuron zero, and any critical point outside this flat region serves as a global minimum. This insight sheds light on single neuron training; under the assumption~\eqref{Vardi} on $F$, if we choose an initial guess for a monotone training algorithm below the flat region, then the algorithm converges to a global minimum.

However, the assumption~\eqref{Vardi} is quite strong and not applicable to our case.
Fortunately, we are still able to obtain a somewhat similar result to \cref{Prop:Vardi} for the general $F$.
In \cref{Thm:local}, we show that any nontrivial critical point has a lower energy than the flat energy region consisting of zero neurons.

\begin{theorem}
\label{Thm:local}
In the single neuron problem~\eqref{local_general}, any critical point $\{ \omega_i^*, b_i^* \} \in \mathbb{R}^{d+1}$ such that $\sigma (\omega_i^* \cdot x + b_i^*) \not\equiv 0$ on $\Omega$ satisfies $E_{i}(\omega_i^*, b_i^*) < E_{i}(\mathbf{0}, 0)$.
\end{theorem}
\begin{proof}
Since $\sigma (\omega_i^* \cdot x + b_i^*)$ is nontrivial on $\Omega$, the set
\begin{equation*}
    D = \left\{ x \in \Omega : \omega_i^* \cdot  x + b_i^* > 0 \right\}
\end{equation*}
has nonzero measure.
Note that $\sigma (\omega_i^* \cdot x + b_i^*) = \omega_i^* \cdot x + b_i^*$ and $\sigma' (\omega_i^* \cdot x + b_i^*) = 1$ on $D$.
Since $\nabla E_{i } (\omega_i^*, b_i^*) = \mathbf{0}$, we obtain
\begin{multline}
    \label{E_loc_grad}
    \mathbf{0} = \nabla E_{i} (\omega_i^*, b_i^*) 
    = \int_{\Omega} \left( \mathcal{A} \sigma (\omega_i^* \cdot x + b_i^*) - F(x) \right) \sigma' ( \omega_i^* \cdot  x + b_i^* ) \begin{bmatrix} x\\1 \end{bmatrix} \,dx \\
    = \int_{D} \left( \mathcal{A} (\omega_i^* \cdot x + b_i^*) - F(x) \right) \begin{bmatrix} x\\1 \end{bmatrix} \,dx,
\end{multline}
where the operator $\mathcal{A}$ was given in~\eqref{cA} and the integrals are done entrywise.
It follows that
\begin{equation*}
\begin{split}
    E_{i}&(\omega_i^*, b_i^*)
    = \frac{1}{2} \int_{D} \left[  \mathcal{A} (\omega_i^* \cdot x + b_i^* ) \right] (\omega_i^* \cdot x + b_i^* ) \,dx - \int_{D} F(x) (\omega_i^* \cdot x + b_i^* ) \,dx \\
    &= \begin{bmatrix} \omega_i^* \\ b_i^* \end{bmatrix} \cdot \int_{D} \left[ \mathcal{A} (\omega_i^* \cdot x + b_i^* ) - F(x) \right] \begin{bmatrix} x \\ 1 \end{bmatrix} \,dx - \frac{1}{2} \int_D \left[ \mathcal{A} (\omega_i^* \cdot x + b_i^* ) \right] (\omega_i^* \cdot x + b_i^* ) \,dx\\
    &\stackrel{\eqref{E_loc_grad}}{=} - \frac{1}{2} \int_D \left[ \mathcal{A} (\omega_i^* \cdot x + b_i^* ) \right] (\omega_i^* \cdot x + b_i^* ) \,dx \\
    &< 0 = E_{i} (\mathbf{0}, 0),
\end{split}
\end{equation*}
which completes the proof.
\end{proof}

\Cref{Thm:local} implies that, if we find a local minimum of~\eqref{local_general} outside the flat region of zero neurons by using some training algorithm, then that local minimum must have a lower energy than the zero neuron.
Consequently, this results in the corresponding neuron being activated, indicating its positive contribution to the training process.

Thanks to \cref{Thm:local}, we can ensure that choosing an initial guess $\{ \omega_i^{(0)}, b_i^{(0)} \}$  such that $E_{i} (\omega_i^{(0)}, b_i^{(0)}) < E_{i} (\mathbf{0})$ makes the training algorithm to converge to a good local minimum that prevents the neuron to be on flat energy regions.
That is, the $\{ \omega_i, b_i \}$-minimization problems of NPSC avoid flat regions.

Various optimization algorithms including first- and second-order methods can be used to solve~\eqref{local_general}.
For our test problems, we propose to use the Levenberg--Marquardt algorithm~\cite{Marquardt:1963}.
The major computational cost of each iteration of the Levenberg--Marquardt algorithm is to solve a linear system with $d+1$ unknowns; it is not time-consuming when $d+1$ is not big.
Furthermore, the Levenberg--Marquardt algorithm, which does not require explicit assembly of the Hessian, converges to a local minimum with the superlinear convergence rate~\cite{YF:2001}, hence, it is much faster than other first-order methods.

\subsection{Adjustment of parameters}
\label{Subsec:Adj}
Here, we discuss the necessity of adjusting parameters during the training.
In~\cite[Theorem~3.1]{VYS:2021}, it was shown that a neuron is initialized as zero with probability close to half if we employ a usual random initialization scheme~\cite{HZRS:2015}.
Such a zero neuron makes the functions $\{ \sigma (\omega_i \cdot x + b_i ) \}_{i=1}^n$ linearly dependent.
Meanwhile, linear dependence may occur by neurons whose nodal points are outside $\Omega$; such scenarios are summarized in \cref{Prop:dependence}.

\begin{proposition}
\label{Prop:dependence}
In the neural network~\eqref{NN}, we have the following:
\begin{enumerate}
    \item If a neuron satisfies $b_i \leq - \max_{x \in \overline{\Omega}} \omega_i \cdot x$, then it is zero.
    \item If there are more than $d+1$ neurons such that $b_i \geq - \min_{x \in \overline{\Omega}} \omega_i \cdot x$, then the functions $\sigma (\omega_i \cdot x + b_i)$ corresponding to these neurons are linearly dependent on $\Omega$.
\end{enumerate}
\end{proposition}
\begin{proof}
If $b_i \leq - \max_{x \in \overline{\Omega}} \omega_i \cdot x$ for some $i$, then we have
\begin{equation*}
    \omega_i \cdot x + b_i \leq \max_{x \in \overline{\Omega}} \omega_i \cdot x + b_i \leq 0
    \quad \text{ in }\Omega.
\end{equation*}
Hence, $\sigma (\omega_i \cdot x + b_i) =0 $ on $\Omega$, which means that the corresponding neuron is zero.

Now, we suppose that $b_i \geq - \min_{x \in \overline{\Omega}} \omega_i \cdot x$ for $1 \leq i \leq d +2$.
Then we have
\begin{equation*}
    \omega_i \cdot x + b_i \geq \min_{x \in \overline{\Omega}} \omega_i \cdot x + b_i \geq 0
    \quad \text{ in }\Omega.
\end{equation*}
This implies that $\sigma (\omega_i \cdot x + b_i) = \omega_i \cdot x +b_i$ on $\Omega$, i.e., each $\sigma (\omega_i \cdot x + b_i)$ is linear on $\Omega$.
Since the dimension of the space of all linear functions on $\Omega$ is $d+1$, $\{ \sigma (\omega_i \cdot x + b_i) \}_{i=1}^{d+2}$ are linearly dependent on $\Omega$.
\end{proof}

Linear dependence of the functions $\{ \sigma (\omega_i \cdot x + b_i ) \}_{i=1}^n$ in the neural network~\eqref{NN} should be avoided because it means that some neurons do not contribute to the approximability of~\eqref{NN}.
\cref{Prop:dependence} motivates us to consider an adjustment procedure for $\{\omega^{(k)}, b^{(k)}\}$ at each iteration of NPSC to avoid linear dependence of neurons.
\cref{Alg:adj} summarizes the adjustment procedure for a given $\{\omega, b\} \in W$.

\begin{algorithm}[H]
\caption{Adjustment for $\{\omega, b \}$ in \cref{Alg:NPSC}}
\begin{algorithmic}[]
\label{Alg:adj}
\STATE Set $m = 0$.
\FOR{$i = 1, \dots, n$}
\STATE Set $\displaystyle
\omega_i \leftarrow \frac{\omega_i}{|\omega_i|}$ and $\displaystyle b_i \leftarrow \frac{b_i}{|\omega_i|}$.
\IF {$\displaystyle b_i \leq -\max_{x \in \overline{\Omega}} \omega_i \cdot x$}
\STATE Reset $b_i \in \mathbb{R}$ such that $\displaystyle b_i \sim \operatorname{Uniform} \left( -\max_{x \in \overline{\Omega}} \omega_i \cdot x, -\min_{x \in  \overline{\Omega}} \omega_i \cdot x \right) $.
\ENDIF
\IF {$\displaystyle b_i \geq -\min_{x \in  \overline{\Omega}} \omega_i \cdot x $}
\IF {$m > d+1$}
\STATE Reset $b_i \in \mathbb{R}$ such that $\displaystyle b_i \sim \operatorname{Uniform} \left( -\max_{x \in \overline{\Omega}} \omega_i \cdot x, -\min_{x \in  \overline{\Omega}} \omega_i \cdot x \right) $.
\ELSE
\STATE $m \leftarrow m+1$
\ENDIF
\ENDIF
\ENDFOR
\end{algorithmic}
\end{algorithm}

Since $\omega_i$ determines the direction of a function $\sigma (\omega_i \cdot x + b_i)$, we may normalize $\{\omega_i, b_i\}$ so that $|\omega_i| = 1$ in \cref{Alg:adj}.
Then, for each $b_i$ that is not on the interval between the minimum and maximum of $\omega_i \cdot x$, we relocate it to the interval.
This relocation step helps avoid linear dependence in neurons described in \cref{Prop:dependence}.

In \cref{Alg:adj}, evaluations of the extrema of $\omega_i \cdot x$ on $\overline{\Omega}$ are simple linear programs, and hence can be done efficiently by conventional algorithms for linear programming~\cite{NW:2006}.
In particular, if $\Omega$ is a polyhedral domain, we can compute $\omega_i \cdot x$ at all the vertices of $\Omega$ and take the extrema among them.

\subsection{Backtracking for learning rates}
\label{Subsec:Backt}
We utilize a backtracking scheme, discussed in \cref{Alg:NPSC}, to find a suitable learning rate $\tau_k$. This approach relieves us from the burden of tuning the learning rate, and also improve the convergence of NPSC.
It was recently shown in~\cite{Park:2022} that parallel subspace correction methods for convex optimization problems can be accelerated by adopting a backtracking scheme.
Although, due to the nonconvexity and nonsmoothness of the loss function $E$ of~\eqref{model}, we are not able to adopt the backtracking scheme proposed in~\cite{Park:2022} or the conventional backtracking schemes such as~\cite{CC:2019,SGB:2014} directly to \cref{Alg:NPSC}, 
a simple but effective backtracking scheme for finding $\tau_k$ presented in \cref{Alg:backt} can still be used.
By allowing adaptive increase and decrease of $\tau_k$ along the epochs, the convergence rate of NPSC is improved, as shown in \cref{App:AblationStudies}.

\begin{algorithm}[H]
\caption{Backtracking scheme to find $\tau_k$ in \cref{Alg:NPSC}}
\begin{algorithmic}[]
\label{Alg:backt}
\STATE Choose a minimum learning rate $\tau_{\min} = 10^{-12}$.
\STATE $\tau_k \leftarrow 2\tau_{k-1}$
\REPEAT
\FOR{$i = 1, \dots, n$ \textbf{in parallel}}
\STATE $\displaystyle
\hat{\omega}_i = (1- \tau_k) \omega_i^{(k)} + \tau_k \omega_i^{(k+\frac{1}{2})}
$
\STATE $\displaystyle
\hat{b}_i = (1- \tau_k) b_i^{(k)} + \tau_k b_i^{(k+\frac{1}{2})}
$
\ENDFOR
\IF {$\displaystyle E (\{ a^{(k+1)}, \hat{\omega}, \hat{b} \}) > E(\{ a^{(k+1)}, \omega^{(k)}, b^{(k)} \})$}
\STATE $\tau_k \leftarrow \tau_k / 2$
\ENDIF
\UNTIL{$\displaystyle E (\{ a^{(k+1)}, \hat{\omega}, \hat{b} \}) \leq E(\{ a^{(k+1)}, \omega^{(k)}, b^{(k)} \})$ or $\tau_k < \tau_{\min}$}
\end{algorithmic}
\end{algorithm}

\section{Numerical experiments}
\label{Sec:Numerical}
In this section, we present numerical results of NPSC applied to function approximation problems and PDEs of the form~\eqref{model}.
The following algorithms are compared with NPSC in our numerical experiments: gradient descent~(GD), Adam~\cite{KB:2015}, hybrid least-squares/gradient descent~(LSGD)~\cite{CGPPT:2020}.
All the algorithms were implemented in ANSI C with OpenMPI compiled by Intel C++ Compiler.
They were executed on a computer cluster equipped with multiple Intel Xeon SP-6148 CPUs~(2.4GHz, 20C) and the operating system CentOS~7.4 64bit.

In all the experiments, we use the He initialization~\cite{HZRS:2015},
that is, we set $a_i \sim N(0, 2/n)$, $\omega_i \sim (N(0, 2/d))^{d}$, and $b_i \sim (N(0, 2/d))$ in~\eqref{model}.
All the numerical results presented in this section are averaged over 10 random initializations.
As shown in \cref{App:AblationStudies}, the performances of conventional training algorithms can be improved by utilizing the backtracking scheme presented in \cref{Alg:backt}.
Hence, for GD, Adam, and LSGD, we employ \cref{Alg:backt} to find learning rates.
For $a$-minimization problems~\eqref{a_min_linear}, we use the least square method~(as implemented in the LAPACK~\cite{ABB+:1999} \texttt{gelsd} subroutine) in LSGD, and the preconditioned conjugate gradient method~(PCG) with the preconditioner $P$ in \cref{Thm:preconditioner} in NPSC.  The stopping criterion for PCG is
\begin{equation}
    \label{PCG_stop}
    \frac{\| Ka^{(k)} - \beta \|_{\ell^2}}{\|Ka^{(0)} - \beta \|_{\ell^2}} < 10^{-10}.
\end{equation}
Finally, $\{ \omega_i, b_i \}$-minimization problems~\eqref{local_general} in NPSC are solved by the Levenberg--Marquardt algorithm~\cite{Marquardt:1963} with the stopping criterion
\begin{equation*}
    \frac{| E_{i}^{(n-1)} - E_{i}^{(n)} |}{| E_{i}^{(n)} | }  < 10^{-10}.
\end{equation*}
MPI parallelization is applied to NPSC in a way that each $\{ \omega_i, b_i \}$-minimization problem is assigned to a single processor and solved in parallel.

\subsection{\texorpdfstring{$L^2$}{L2}-function approximation problems}
\begin{figure}
  \centering
  \includegraphics[width=0.77\hsize]{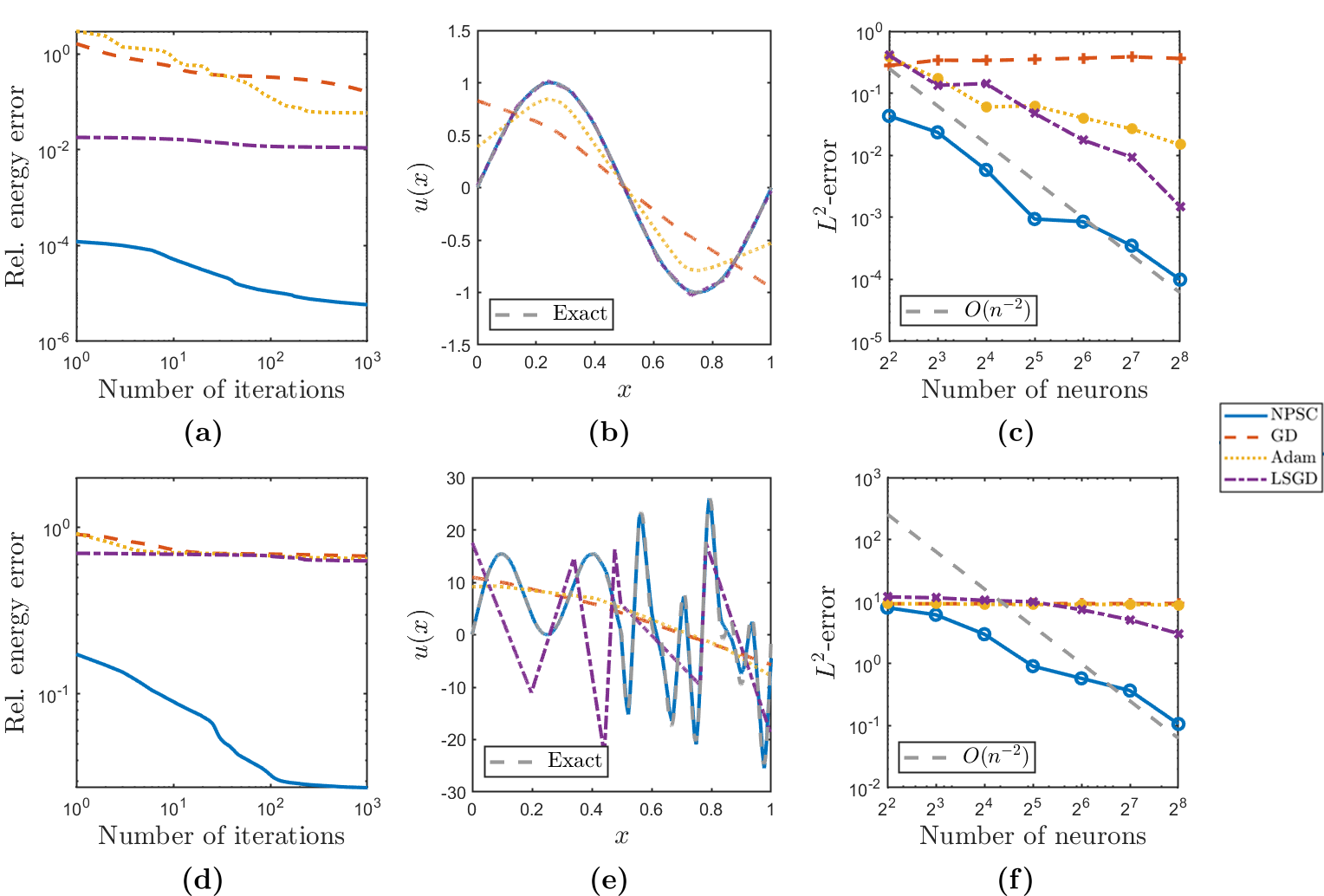}
  \caption{Numerical results for the function approximation problems {\rm\textbf{(a--c)}}~\eqref{Ex1} and {\rm\textbf{(d--f)}}~\eqref{Ex2}.
  {\rm\textbf{(a, d)}} Decay of the relative energy error $\frac{E(\theta^{(k)}) - E^*}{|E^*|}$ in various training algorithms~($n = 2^5$).
  {\rm\textbf{(b, e)}} Exact solution and its approximations~($n = 2^5$, $10^3$ epochs).
  {\rm\textbf{(c, f)}} $L^2$-errors with respect to the number of neurons~($10^3$ epochs).
  }
  \label{Fig:Ex12}
\end{figure}

If we set the bilinear form in~\eqref{model} to be~\eqref{L2}, we obtain the $L^2$-function approximation problem, which is the most elementary instance of~\eqref{model}. The solution of~\eqref{model} in this case is the best $L^2$-approximation of $f$ found in $\Sigma_n$.

\subsubsection*{Test 1}
As our first example, we consider $L^2$-approximation~\eqref{L2} for a sine function; we set
\begin{equation}
\label{Ex1}
\Omega = (0, 1) \subset \mathbb{R}, \quad
f(x) = \sin 2\pi x
\end{equation}
in~\eqref{model}.
For numerical integration, we use the trapezoidal rule on 10,000 uniformly sampled points; see \cref{App:Integration}.
\cref{Fig:Ex12}(a) plots the relative energy errors $\frac{E(\theta^{(k)}) - E^*}{|E^*|}$ obtained by NPSC, GD, Adam, and LSGD per epoch, where $E^*$ is the loss corresponding to the exact solution.
It is clear from the loss decay that NPSC outperforms all the other methods.
The exact solution of~\eqref{Ex1} and its approximations obtained by $10^3$ epochs of the training algorithms with $2^5$ neurons are depicted in \cref{Fig:Ex12}(b).
One can readily observe that the NPSC result is the most accurate among all.
The $L^2$-errors between the exact solution and its approximations for various numbers of neurons are presented in \cref{Fig:Ex12}(c).
Only the NPSC result seems to be comparable to $\mathcal{O}(n^{-2})$, the theoretical optimal rate derived in~\cite{SX:2022c}.

\subsubsection*{Test 2}
We revisit~\eqref{Ex2_pre} as the second example, a highly oscillatory instance of~\eqref{L2}:
\begin{equation}
\label{Ex2}
\Omega = (0, 1) \subset \mathbb{R}, \quad
f(x) = \begin{cases}
10 \left( \sin 2\pi x + \sin 6 \pi x \right), & \quad \textrm{ if } 0 \leq x < \frac{1}{2}, \\
10 \left( \sin 8\pi x + \sin 18 \pi x + \sin 26 \pi x \right), & \quad \textrm{ if } \frac{1}{2} \leq x \leq 1.
\end{cases}
\end{equation}
A similar problem appeared in~\cite{CLL:2020}.
We note that it was demonstrated in~\cite{KGZKM:2021} that training for complex functions like~\eqref{Ex2} is a much harder task than training for simple functions.
Same as in~\eqref{Ex1}, we use the trapezoidal rule on 10,000 uniformly sampled points for numerical integration.
\cref{Fig:Ex12}(d--f) present numerical results for the problem~\eqref{Ex2}.
In \cref{Fig:Ex12}(d), NPSC shows stable decay of the loss while the losses of other methods are stagnant along the epochs.
In \cref{Fig:Ex12}(e), one can observe that the NPSC result captures both low- and high-frequency parts of the target function well, while the other ones capture the low-frequency part only~\cite{XZLXZ:2020}.
Moreover, as shown in \cref{Fig:Ex12}(f), the $L^2$-error of the NPSC decreases when the number of neurons increases, while those of the other algorithms seem to stagnate.
This highlights the robustness of NPSC for oscillatory functions.

\subsection{Elliptic PDEs}
\begin{figure}
  \centering
  \includegraphics[width=\hsize]{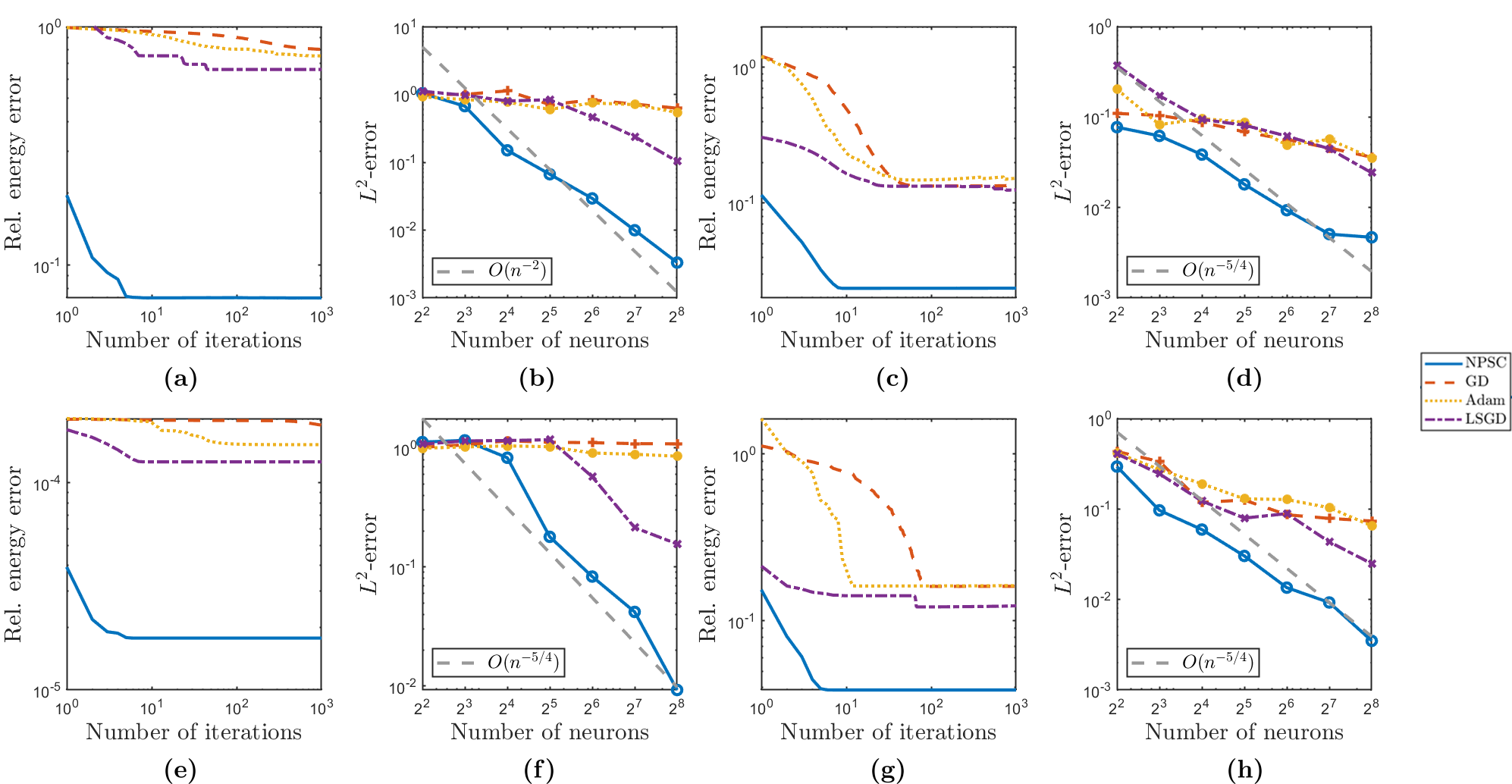}
  \caption{Numerical results for the elliptic PDEs {\rm\textbf{(a, b})}~\eqref{Ex3}, {\rm\textbf{(c, d)}}~\eqref{Ex4}, {\rm\textbf{(e, f})}~\eqref{Ex5}, and {\rm\textbf{(g, h})}~\eqref{Ex6}.
  {\rm\textbf{(a, c, e, g)}} Decay of the relative energy error $\frac{E(\theta^{(k)}) - E^*}{|E^*|}$ in various training algorithms~($n = 2^5$).
  {\rm\textbf{(b, d, f, h)}}~$L^2$-errors with respect to the number of neurons~($10^3$ epochs).
  }
  \label{Fig:Ex34}
\end{figure}

Next, we consider the case
\begin{equation}
    \label{Neumann}
    V = H^1 (\Omega), \quad
    a(u,v) = \int_{\Omega} \left( \alpha \nabla u \cdot \nabla v + uv \right) \,dx, \quad u,v \in V
\end{equation}
in the problem~\eqref{model}, where $H^1 (\Omega)$ is the usual Sobolev space consisting of $L^2 (\Omega)$-functions with square-integrable gradients, and $\alpha \in L^{\infty} (\Omega)$ is a function satisfying $\alpha \geq \alpha_0$ in $\Omega$ for some $\alpha_0 > 0$.
It is well-known that~\eqref{model} is the Galerkin approximation of the weak formulation of the following elliptic PDE on $\Sigma_n$~\cite{Xu:2020}:
\begin{equation*}
    - \nabla \cdot (\alpha \nabla u) + u = f \quad \text{ in } \Omega, \quad
    \frac{\partial u}{\partial n} = 0 \quad \text{ on } \partial \Omega,
\end{equation*}
where $\partial u / \partial n$ denotes the normal derivative of $u$ to $\partial \Omega$.
In the following, we carry out four different numerical experiments in both one and two dimensions, using constant and oscillating coefficient functions $\alpha$.

\subsubsection*{Test 1}
In this test, we set
\begin{equation}
\label{Ex3}
\Omega = (0, 1) \subset \mathbb{R}, \quad
\alpha = 1, \quad
f(x) = (1 + 9\pi^2) \cos 3\pi x + (1+121\pi^2) \cos 11\pi x
\end{equation}
in~\eqref{Neumann}. The exact solution is $u(x) = \cos 3\pi x + \cos 11\pi x$.
We can observe in \cref{Fig:Ex34}(a) that NPSC achieves a superior level of accuracy that other algorithms do not reach even with a large number of epochs.
As shown in \cref{Fig:Ex34}(b), $L^2$-errors of the NPSC results decrease as the number of neuron increases, and the decreasing rate is comparable to the theoretical optimal rate in~\cite{SX:2022c}.

\subsubsection*{Test 2}
We consider the following two-dimensional example for~\eqref{Neumann}:
\begin{equation}
    \label{Ex4}
    \Omega = (0,1)^2 \subset \mathbb{R}^2, \quad
    \alpha = 1, \quad
    f(x_1, x_2) = (1 + 2\pi^2) \cos \pi x_1 \cos \pi x_2,
\end{equation}
whose exact solution is $u(x_1, x_2) = \cos \pi x \cos \pi y$.
Since the preconditioner in \cref{Thm:preconditioner} is applicable for one dimension only, in this example, $a$-minimization problems are solved by $n$ iterations of the unpreconditioned conjugate gradient method.
We use the quasi-Monte Carlo method~\cite{Caflisch:1998} with 10,000 sampling points for numerical integration; see \cref{App:Integration} for details.
It is verified by the numerical results for~\eqref{Ex4} presented in \cref{Fig:Ex34}(c, d) that NPSC outperforms other methods and provides reasonable $L^2$-errors in high-dimensional problems as well.

\subsubsection*{Test 3}
Now, we examine examples where the coefficient function $\alpha$ exhibits oscillatory behavior.
We consider the following example:
\begin{equation}
    \label{Ex5}
    \Omega = (0,1) \subset \mathbb{R}, \quad
    \alpha(x) = \sin(6\pi x) + 2, \quad u(x) = \cos(12\pi x) + \cos(16\pi x).
\end{equation}
The numerical results are plotted in \cref{Fig:Ex34}(e, f).
Similar to the case of constant $\alpha$, we observe that NPSC outperforms the other algorithms in terms of accuracy.

\subsubsection*{Test 4}
Lastly, we consider the following two-dimensional example for~\eqref{Neumann} with an oscillatory coefficient function $\alpha$:
\begin{equation}
    \label{Ex6}
    \Omega = (0,1)^2 \subset \mathbb{R}^2, \quad
    \alpha(x_1, x_2) = \sin(6\pi x_1) + 2, 
    \quad
    u(x_1, x_2) = \cos (\pi x_1) \cos (\pi x_2 ).
\end{equation}
For this example with $2^8$ neurons, we utilize 20,000 sampling points for numerical integration due to the oscillatory nature of the problem that necessitates a more accurate quadrature.
As shown in \cref{Fig:Ex34}(g, h), the proposed NPSC consistently outperforms other algorithms in terms of accuracy.

\begin{remark}
\label{Rem:ReLUk}
When we train a neural network~\eqref{NN} for solving the PDE~\eqref{Neumann} with a gradient-based method, we have to evaluate the second-order derivative of $\sigma (x)$, which is the Dirac delta function.
In our experiments, we simply ignore Dirac delta terms in numerical integration.
This motivates us to consider high-order activation functions like ReLU$^k$~\cite{SX:2022b} as a future work.
\end{remark}

\subsection{Ablation studies}

We also validate the key components of the proposed NPSC method, namely, the optimal preconditioner for the linear layer~(\cref{Thm:preconditioner}), the Levenberg--Marquardt algorithm for $\{ \omega_i, b_i \}$-minimization problems, the adjustment step for parameters~(\cref{Alg:adj}), and the backtracking scheme for learning rates~(\cref{Alg:backt}) by conducting ablation studies. The numerical results are discussed in~\cref{App:AblationStudies}.


\section{Conclusion}
\label{Sec:Conclusion}
In this paper, we proposed NPSC for the finite neuron method with ReLU neural networks.
Separately designing efficient solvers for the linear layer and each neuron and training them alternately, NPSC yields accurate results for function approximation problems and PDEs. With the NPSC method, one can separately design solvers for the linear and nonlinear layers.
We note that, while the proposed optimal preconditioner is only available for one-dimensional problems, the parallel neuron-wise optimization of parameters in the nonlinear layer is not limited by dimensions.

\subsection{Limitations and future directions}
This paper leaves us several interesting and important topics for future research.

In terms of mathematical theory, although NPSC adopts the well-established framework of subspace correction methods, its rigorous convergence analysis remains an open question due to the nonconvexity of the model.

Regarding accuracy, while NPSC provides more accurate solutions than conventional training algorithms, its performance may not surpass that of classical adaptive finite element methods unless the parameters are initialized sufficiently well.
This suggests the necessity of developing an improved training algorithm that achieves superior accuracy to classical methods regardless of initialization of parameters, which poses a significant challenge given the nonconvex nature of training neural networks.

Concerning the range of applications, addressing high-dimensional problems with NPSC requires the design of effective preconditioners for optimizing the linear layer.
Generalizing the one-dimensional preconditioner proposed in this paper to higher dimensions poses a nontrivial and challenging task.
Another critical challenge in solving high-dimensional problems is numerical quadrature.
Approximating integrals in high dimensions demands a substantial number of integration points, resulting in a significant computational cost.

\appendix
\section{Implementation details for the optimal preconditioner}
\label{App:Implementation}
In this appendix, we discuss the computational aspects of the proposed preconditioner.
Although the preconditioner $P$ defined in~\eqref{P} seems a bit complicated at the first glance, its computation requires only a cheap cost.
We provide a detailed explanation of how to apply the preconditioner $P$ efficiently.

Three nontrivial parts in the preconditioner $P$ are the inverses of the matrices $B_1$, $\overline{K}_{\phi}$, and $R \overline{P}^{-1} R^{\mathrm{T}}$.
First, we consider how to compute $B_1^{-1} \alpha$ when a vector $\alpha \in \mathbb{R}^{n_{\Omega}+2}$ is given.
We solve a linear system $B_1 \beta = \alpha$ in order to obtain $B_1^{-1} \alpha$.
Thanks to the sparsity pattern of $B_1$, this linear system can be solved directly by the following $\mathcal{O}(n)$ elementary arithmetic operations:
\begin{equation*}
    \beta_i = \begin{cases}
    \dfrac{1}{(B_1)_{ii}} \left( \alpha_i - \dfrac{(B_1)_{i,n_{\Omega}+1}}{(B_1)_{n_{\Omega}+1,n_{\Omega}+1}}\alpha_{n_{\Omega}+1} - \dfrac{(B_1)_{i,n_{\Omega}+2}}{(B_1)_{n_{\Omega}+2,n_{\Omega}+2}}\alpha_{n_{\Omega}+2} \right),
    &  \hspace{-0.1cm}\text{if } 1 \leq i \leq n_{\Omega}, \\
    \dfrac{\alpha_i}{(B_1)_{ii}},
    &  \hspace{-0.1cm}\text{otherwise.}
    \end{cases}
\end{equation*}

Secondly, we consider how to compute $\overline{K}_{\phi}^{-1} \alpha$ when a vector $\alpha \in \mathbb{R}^{n_{\Omega}+2}$ is given.
In most applications, the bilinear form $a(\cdot, \cdot)$ is defined in terms of differential operators as in~\eqref{Neumann}.
This implies that the matrix $\overline{K}_{\phi}$ is tridiagonal, so that $\overline{K}_{\phi}^{-1} \alpha$ can be obtained by the Thomas algorithm~(see, e.g.,~\cite[Section~9.6]{Higham:2002}), which requires only $\mathcal{O}(n)$ elementary arithmetic operations. 

Finally, we discuss how to compute $(R \overline{P}^{-1} R^{\mathrm{T}})^{-1} \alpha$ for $\alpha \in \mathbb{R}^n$. Similarly, one can compute it by solving a linear system
\begin{equation*}
R \overline{P}^{-1} R^{\mathrm{T}} \beta = \alpha.
\end{equation*}
By the definition~\eqref{R} of $R$, we have
\begin{equation*}
    \overline{P}^{-1}R^{\mathrm{T}} \beta = \begin{bmatrix} \alpha_{1:n_{\Omega}} \\ \gamma \end{bmatrix}
\end{equation*}
for some $\gamma \in \mathbb{R}^2$, where $1:n_{\Omega}$ means ``from the first to $n_{\Omega}$th entries."
One can readily deduce that $\gamma$ is determined by the following two relations:
\begin{subequations}
\label{gamma}
\begin{align}
\label{gamma1}
\widetilde{R} \gamma = \alpha_{( n_{\Omega} + 1 ) : n }, \\
\label{gamma2}
\left( \overline{P} \begin{bmatrix} \alpha_{1:n_{\Omega}} \\ \gamma \end{bmatrix} \right)_{(n_{\Omega}+1):(n_{\Omega}+2)} \in \operatorname{ran} \widetilde{R}^{\mathrm{T}}.
\end{align}
\end{subequations}
Since $\widetilde{R} \in \mathbb{R}^{(n - n_{\Omega} ) \times 2}$,~\eqref{gamma1} and~\eqref{gamma2} are written as $n - n_{\Omega}$ and $2 - (n - n_{\Omega})$ linear equations, respectively.
That is, $\gamma$ can be determined by solving the system~\eqref{gamma} of two linear equations.

In what follows, we deal with how to solve the linear system~\eqref{gamma} algebraically in detail.
We write $\gamma = [\gamma_1, \gamma_2]^{\mathrm{T}} \in \mathbb{R}^2$.
By~\eqref{nodes2}, the number of interior nodal points $n_{\Omega}$ is either $n-2$, $n-1$, or $n$.
First, we consider the case $n_{\Omega} = n-2$.
In this case,~\eqref{gamma2} obviously holds and~\eqref{gamma1} reads as
\begin{subequations}
\label{gamma1_detail}
\begin{align}
    \label{gamma1_detail1}
    \psi_{n_{\Omega} + 1}(1) \gamma_1 + \psi_{n_{\Omega} + 1}(0) \gamma_2 &= \alpha_{n_{\Omega} + 1}, \\
    \label{gamma1_detail2}
    \psi_{n_{\Omega} + 2}(1) \gamma_1 + \psi_{n_{\Omega} + 2}(0) \gamma_2 &= \alpha_{n_{\Omega} + 2}.
\end{align}
\end{subequations}
Hence, one can find $\gamma$ by solving~\eqref{gamma1_detail} directly.

Next, we assume that $n_{\Omega} = n-1$.
Then~\eqref{gamma1} and~\eqref{gamma2} read as~\eqref{gamma1_detail1} and
\begin{equation}
    \label{gamma2_detail}
    \overline{P}_{(n_{\Omega} + 1):(n_{\Omega} + 2), 1:n_{\Omega}} \alpha_{1:n_{\Omega}} + \overline{P}_{(n_{\Omega} + 1):(n_{\Omega} + 2), (n_{\Omega}+1):(n_{\Omega}+2)} \gamma
    \in \operatorname{span} \left\{
    \begin{bmatrix}
    \psi_{n_{\Omega} + 1} (1) \\ \psi_{n_{\Omega} + 1} (0)
    \end{bmatrix}
    \right\},
\end{equation}
respectively.
Since~\eqref{gamma2_detail} is equivalent to
\begin{multline}
    \label{gamma2_detail_re}
    \left( \overline{P}_{n_{\Omega} + 1, 1:n_{\Omega}} \alpha_{1:n_{\Omega}} + \overline{P}_{n_{\Omega} + 1, (n_{\Omega}+1):(n_{\Omega}+2)} \gamma \right) \psi_{n_{\Omega} + 1}(0) \\
    - \left( \overline{P}_{n_{\Omega} + 2, 1:n_{\Omega}} \alpha_{1:n_{\Omega}} + \overline{P}_{n_{\Omega} + 2, (n_{\Omega}+1):(n_{\Omega}+2)} \gamma \right) \psi_{n_{\Omega} + 1} (1) = 0,
\end{multline}
$\gamma$ can be obtained by solving the linear system consisting of~\eqref{gamma1_detail1} and~\eqref{gamma2_detail_re}.

Finally, we examine the case $n_{\Omega} = n$.
In this case,~\eqref{gamma1} is void and~\eqref{gamma2} reads as
\begin{align*}
\overline{P}_{(n_{\Omega} + 1):(n_{\Omega} + 2), 1:n_{\Omega}} \alpha_{1:n_{\Omega}} + \overline{P}_{(n_{\Omega} + 1):(n_{\Omega} + 2), (n_{\Omega}+1):(n_{\Omega}+2)} \gamma &= 0.
\end{align*}
Therefore, $\gamma$ is found by solving the above linear system.

Now, $\beta$ is then given by
\begin{equation*}
\begin{split}
    \beta_{1:n_{\Omega}} &= \left( \overline{P} \begin{bmatrix} \alpha_{1:n_{\Omega}} \\ \gamma \end{bmatrix} \right)_{1:n_{\Omega}}, \\
    \widetilde{R}^{\mathrm{T}} \beta_{(n_{\Omega}+1):n} &= \left( \overline{P} \begin{bmatrix} \alpha_{1:n_{\Omega}} \\ \gamma \end{bmatrix} \right)_{(n_{\Omega}+1):(n_{\Omega}+2)}.
\end{split}
\end{equation*}
In summary, the optimal preconditioner requires only $\mathcal{O}(n)$ elementary arithmetic operations.

\section{Numerical results for the ablation studies}
\label{App:AblationStudies}
In this appendix, we discuss the numerical results of the ablation studies for the proposed NPSC method.

\subsection{Effect of preconditioning}
\label{App:preconditioning}
\begin{table}
    \centering
    \begin{tabular}{c|c|c|c|c}
      $n$ & GD & Adam &  CG & PCG\\
          \hline 
       $2^4$  & $>500000$  & 105692   & 34 & 2 \\
    $2^5$     & $>500000$  & $>500000$  & 107 & 2 \\
    $2^6$     & $>500000$  & $>500000$  & 310 & 2 \\
    $2^7$     & $>500000$  & $>500000$  & 1018 & 3 \\
    $2^8$     & $>500000$  & $>500000$ & 3470 & 3
    \end{tabular}
    \caption{Number of iterations required for the gradient descent method~(GD), Adam, conjugate gradient method~(CG), and preconditioned conjugate gradient method~(PCG) to solve the $a$-minimization problems~\eqref{a_min} corresponding to the $L^2$-function approximation problem~\eqref{Ex2} with different numbers of neurons $n$.} 
    \label{Table:preconditioner1}
\end{table}

\begin{table}
    \centering
    \begin{tabular}{c|c|c|c|c}
      $n$ & GD & Adam &  CG & PCG\\
          \hline 
       $2^4$  & 5705  &1430  & 19 & 3 \\
    $2^5$    & 22678  &2847  & 37 & 2 \\
    $2^6$    &90312  & 5989 & 59 & 4 \\
    $2^7$    &360335  &17233 & 103 & 4  \\
    $2^8$    & $>500000$ & 55314 & 173 & 5 
    \end{tabular}
    \caption{Number iterations required for the gradient descent method~(GD), Adam, conjugate gradient method~(CG), and preconditioned conjugate gradient method~(PCG) to solve the $a$-minimization problems~\eqref{a_min} corresponding to the elliptic PDE~\eqref{Ex3} with different numbers of neurons $n$.} 
    \label{Table:preconditioner2}
\end{table}

We present some numerical results to highlight the computational efficiency of the preconditioner $P$.
The $a$-minimization problems~\eqref{a_min_linear} appearing in training of neural networks that solve the $L^2$-function approximation problem~\eqref{Ex2} and the elliptic boundary value problem~\eqref{Ex3} are considered.
We fix the parameters $\omega$ and $b$ in the nonlinear layer as follows:
\begin{equation*}
\omega_i = 1,\quad b_i = -\frac{i}{n+1}, \quad 1\leq i\leq n.
\end{equation*}
We compare the numerical performance of GD, Adam, the preconditioned and unpreconditioned conjugate gradient methods solving~\eqref{a_min_linear} in \cref{Table:preconditioner1} and \cref{Table:preconditioner2} with the stopping criterion~\eqref{PCG_stop} for different numbers of neurons.
While the number of unpreconditioned iterations increases rapidly as $n$ increases, the number of preconditioned iterations is uniformly bounded with respect to $n$.
This showcases the effectiveness of our preconditioned presented in \cref{Thm:preconditioner}. 
Since the computational cost of $P$ is cheap, the proposed preconditioner $P$ is numerically efficient.

\subsection{Effect of backtracking for conventional training algorithms}
We present numerical results that show that the backtracking scheme presented in \cref{Alg:backt} is useful not only for the proposed NPSC but also for conventional training algorithms such as GD, Adam~\cite{KB:2015}, and LSGD~\cite{CGPPT:2020}.

\begin{figure}
  \centering
  \includegraphics[width=0.8\hsize]{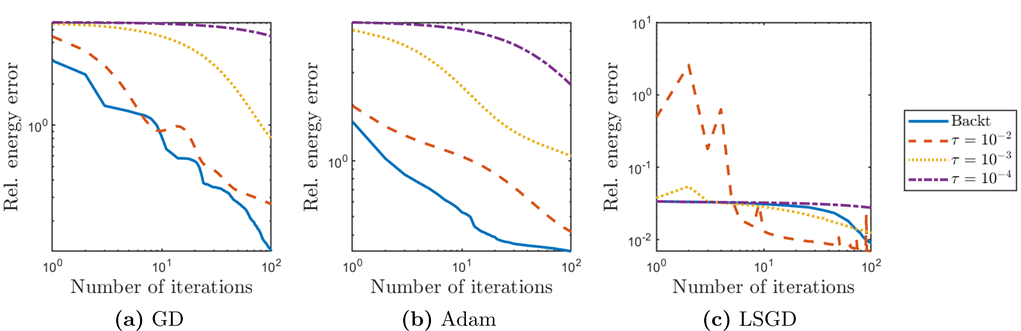}
  \caption{Decay of the relative energy error $\frac{E(\theta^{(k)}) - E^*}{|E^*|}$ in various training algorithms for solving~\eqref{Ex1}.
  ``Backt'' denotes the backtracking scheme presented in \cref{Alg:backt}, and $\tau$ denotes the fixed learning rate.}
  \label{Fig:backt}
\end{figure}

\cref{Fig:backt} plots the relative energy error $\frac{E(\theta^{(k)}) - E^*}{|E^*|}$ of GD, Adam, and LSGD for solving the problem~\eqref{Ex1}, averaged over 10 random initializations, where $k$ denotes the number of epochs and $E^*$ is the energy corresponding to the exact solution of the problem.
The number of neurons used is $2^5$; while we can observe similar results for the other numbers of neurons, we only provide the result of $2^5$ neurons for brevity.
We observe that the algorithms equipped with the backtracking scheme outperform those with constant learning rates $\tau = 10^{-2}$, $10^{-3}$, and $10^{-4}$ in the sense of the convergence rate.
That is, \cref{Alg:backt} seems to successfully find a good learning rate at each iteration of conventional training algorithms as well.
Hence, in \cref{Sec:Numerical}, we employ \cref{Alg:backt} to find learning rates of GD, Adam, and LSGD.

\subsection{Effect of parameter adjustment and Levenberg-Marquardt algorithm}\label{App:param-and-LMA}

\cref{Fig:ablation} depicts numerical comparisons among three algorithms: NPSC, NPSC without \cref{Alg:adj} (parameter adjustment), and NPSC without the Levenberg--Marquardt algorithm.

One can see the variants of NPSC achieve slower convergence rates than NPSC in all the examples. Therefore, both \cref{Alg:adj} and the Levenberg--Marquardt algorithm contribute to the fast convergence of NPSC.
We also note that \cref{Alg:adj} helps NPSC to avoid unstable convergence behaviors like \cref{Fig:ablation}(c, d).

\begin{figure}
  \centering
  \includegraphics[width=\hsize]{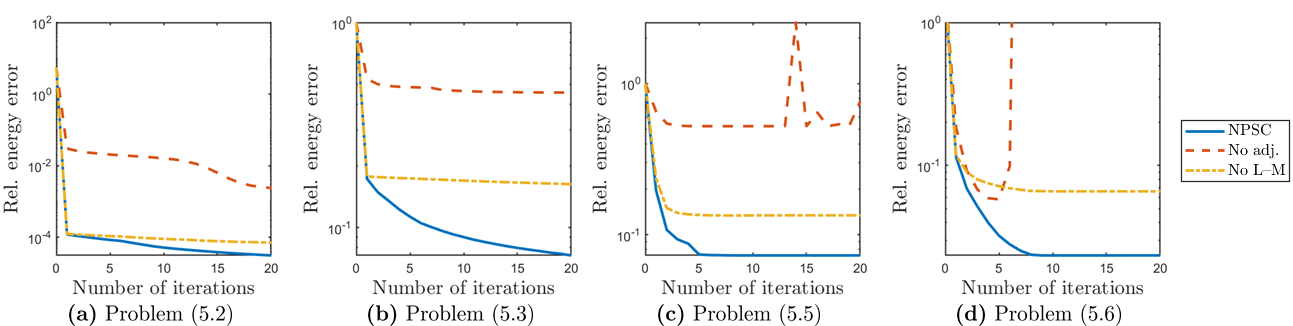}
  \caption{Ablation studies for NPSC on the relative energy error $\frac{E(\theta^{(k)}) - E^*}{|E^*|}$. ``No adj." and ``No L--M'' denote NPSC without the adjustment step and the Levenberg--Marquardt algorithm, respectively.}
  \label{Fig:ablation}
\end{figure}


\section{Numerical integration}
\label{App:Integration}
This appendix is devoted to numerical integration schemes for computing the integral in our model problem~\eqref{model}.
If $d=1$, i.e., if the domain $\Omega = (0, 1) \subset \mathbb{R}$, then we approximate the integral of a function $g(x)$ defined on $\Omega$ by the following simple trapezoidal rule:
\begin{equation}
\label{int_1D}
    \int_{\Omega} g(x) \,dx \approx \sum_{j=1}^{N-1} \frac{g(x_j) + g(x_{j+1})}{2} (x_{j+1} - x_j),
\end{equation}
where $0 = x_1 < x_2 < \dots < x_N = 1$ are $N$ uniform sampling points between $0$ and $1$, i.e., $x_j = \frac{j-1}{N-1}$, $1 \leq j \leq N$.
Approximation properties of the trapezoidal rule~\eqref{int_1D} can be found in standard textbooks on numerical analysis; see, e.g.,~\cite{BFB:2015}.

When $d \geq 2$, we adopt the quasi-Monte Carlo method~\cite{Caflisch:1998} based on Halton sequences~\cite{KW:1997}, which is known to overcome the curse of dimensionality in the sense that approximation error bounds independent of the dimension $d$ are available.
In the quasi-Monte Carlo method, the integral of a function $g(x)$ defined on $\Omega$ is approximated by the average of the function evaluated at $N$ sampling points:
\begin{equation}
\label{int_dD}
    \int_{\Omega} g(x) \,dx \approx \frac{1}{N} \sum_{j=1}^N g(x_j),
\end{equation}
where $\{ x_j \}_{j=1}^N$ is a low-discrepancy sequence in $\Omega$ defined in terms of Halton sequences.
For the sake of description, we assume that $\Omega = (0,1)^d \in \mathbb{R}^d$.
Then the $k$th coordinate of $x_j$~($1 \leq j \leq N$, $1 \leq k \leq d$) is the number $j$ written in $p_k$-ary representation, inverted, and written after the decimal point, where $p_k$ is the $k$th smallest prime number.
For example, if $d = 3$, then the first four points of $\{ x_j \}_{j=1}^N$ is given by
\begin{equation*}
    x_1 = \left( \frac{1}{2}, \frac{1}{3}, \frac{1}{5} \right), \hspace{0.1cm}
    x_2 = \left( \frac{1}{2^2}, \frac{2}{3}, \frac{2}{5} \right), \hspace{0.1cm}
    x_3 = \left( \frac{3}{2^2}, \frac{1}{3^2}, \frac{3}{5} \right), \hspace{0.1cm}
    x_4 = \left( \frac{1}{2^3}, \frac{4}{3^2}, \frac{4}{5} \right).
\end{equation*}
Approximation properties of~\eqref{int_dD} can be found in~\cite{Caflisch:1998}.

\bibliographystyle{siamplain}
\bibliography{refs_NPSC}
\end{document}